\documentclass[12pt]{article}
\usepackage{amsmath}
\usepackage{amsfonts}
\usepackage{amssymb}
\usepackage{mathtools}
\usepackage{amsthm}
\usepackage{tensor}
\usepackage{enumerate}
\usepackage{tikz}
\usepackage{url}
\newtheorem{theorem}{Theorem}
\newtheorem{problem}{Problem}
\newtheorem{proposition}{Proposition}
\newtheorem{corollary}{Corollary}

\newcommand{\D}{\nabla}
\newcommand{\R}{\mathbb{R}}
\newcommand{\C}{\mathbb{C}}
\newcommand{\ol}{\overline}
\newcommand{\ot}{\otimes}
\newcommand{\op}{\oplus}

\newcommand{\SO}{SO}
\newcommand{\SU}{SU}
\renewcommand{\leq}{\leqslant}
\renewcommand{\geq}{\geqslant}
\DeclareMathOperator{\sgn}{sgn}

\author{Shun-ichi Amari\footnote{RIKEN Brain Science Institute, Saitama 351-0198, Japan \tt{amari@brain.riken.jp}}
\and John Armstrong\footnote{Dept.\ of Mathematics, King's College London, UK. \tt{john.1.armstrong@kcl.ac.uk}}
}
\title{Curvature of Hessian manifolds}

\begin{document}

\maketitle

\begin{abstract}
We prove that, in dimensions greater than $2$, the generic metric is not a Hessian metric and find a curvature condition on Hessian metrics in dimensions greater than $3$. In particular we prove that the forms used to define the Pontryagin classes in terms of the curvature vanish on a Hessian manifold. By contrast all analytic Riemannian $2$-metrics are Hessian metrics.
\end{abstract}

\section{Introduction}

A Riemannian metric $g$ is called a Hessian metric if there exist local coordinates such that $g$ can be written as the Hessian of some convex potential function $\phi$. This paper is motivated by the question of determining whether or not a given metric $g$ is Hessian.

Hessian metrics have been shown to play an important role in a variety of applications. For example they arise in the study of optimization \cite{nesterov1994interior}, statistical manifolds \cite{amari} and, via special K\"ahler manifolds, in string theory \cite{hitchin1999moduli,freed1999special}.

A recurring theme in the study of Hessian metrics is that of duality. The Legendre--Fenchel transform provides a basic notion of duality for convex functions \cite{rockafellar} which has many important applications. This duality manifests itself in Hessian geometry in the study of affine connections.

Given any affine connection $\ol{\D}$ on a Riemannian manifold $(M,g)$, we define the ``$g$-dual connection'' $\ol{\D}^*$ by:
\[ g( \ol{\D}_X Y, Z ) = g( Y, \ol{\D}^*_X Z ). \]
The Levi--Civita connection is self dual. A $g$-dually flat structure is a pair of $g$-dual connections which are both flat.
A metric locally admits a $g$-dually flat structure if and only if it is Hessian. To see why, one simply needs to know that the geodesics of $g$-dually flat connection define local coordinates with respect to which $g$ is a Hessian metric and the converse is also true \cite{shima,amari}. This gives a dictionary between Hessian geometry and the geometry of affine connections. This dictionary is not entirely trivial: for example, the duality between $\ol{\D}$ and $\ol{\D}^*$ translates into the Legendre transform of the potential $\phi$. As we shall see,
results which are easy to prove from the perspective of affine connections can be harder to understand from the perspective of Hessian geometry and vice versa.

The issue of determining whether a metric $g$ is a Hessian metric was raised in \cite{furuhata, amari} in the language of $g$-dually flat connections. They posed the following basic questions:
\begin{problem}
Let $(M,g)$ be a Riemannian manifold, does there always exist a dually flat structure on $M$, i.e.\ a pair of $g$-dual flat, torsion-free affine connections?
\end{problem}
\begin{problem}
If the answer to Problem 1 is negative, find conditions and invariants which characterize the spaces for which this is possible.
\end{problem}

We interpret these questions as being essentially local questions. So these problems are equivalent to determining whether $g$ is a Hessian metric. We will show that the answer to Problem 1 is positive in dimension $2$ and negative in dimensions greater than $2$. These results have been found independently by Robert Bryant \cite{bryant} and follow from applications of general Cartan--K\"ahler theory.

We also find an explicit curvature condition which must hold for a dually flat structure to exist in dimensions greater than or equal to $4$ and examine its implications. In particular we will prove that the {\em Pontryagin forms} must vanish on a Hessian manifold. (We define the Pontryagin {\em forms} to be the closed forms defined using polynomials in the curvature tensor which provide representatives for the Pontryagin {\em classes}). Thus we obtain a topological obstruction to the existence of a Hessian metric.

The existence of a Hessian metric does not imply the global existence of a $g$-dually flat structure. For example, quotients of hyperbolic space always admit a Hessian metric but may have non-vanishing Euler characteristic which means that they cannot admit a flat connection. The problem of determining whether a manifold $M$ admits any global dually flat structure was considered in \cite{ay} --- but as we have seen this is quite a different problem to determining if it admits a Hessian metric.

The structure of the paper is as follows. In Section \ref{section:counting} we prove that the generic metric of dimension $\geq 3$ is not a Hessian metric. In Section \ref{section:curvature} we find the curvature obstruction and examine its implications. In Section \ref{section:cartanKahler} we prove that all analytic $2$-metrics are Hessian.

\section{A counting argument for Problem 1}
\label{section:counting}

\begin{theorem}
In dimension $n \geq 3$, a generic Riemannian metric $(M^n,g)$ does not admit a compatible dually flat structure, even locally.
\end{theorem}
\begin{proof}
If $(M^n,g)$ admits a dually flat structure then, in the neighbourhood of any point $p$, there exist local coordinates $x:M^n \longrightarrow {\mathbb{R}}^n$ and a potential $\phi$ such that in these coordinates the metric satisfies:
\[
g_{ij} = \frac{\partial^2 \phi}{\partial x_i \partial x_j}.
\]
Thus the $k$-jet of $g$ at $p$ is determined by the $(k+2)$-jet of $x$ and $\phi$ at $p$.

If we fix some reference coordinates around $p$, then the coordinate function $x$ is defined by $n$ real valued functions in $n$ variables. The coordinate $\phi$ is a single real valued function of $n$ variables. Thus the dimension of the space of $(k+2)$-jets of $(x,\phi)$ at $p$ is equal to:
\[ \dim J_{k+2}(x,\phi):=\sum_{i=0}^{k+2} (n+1) \dim (S^i T_p) =
\sum_{i=0}^{k+2} (n+1)\binom{ n + i - 1}{i}. \]

In these same reference coordinates, the metric $g$ is defined by $\frac{n(n+1)}{2}$ real valued functions. So the space of $k$-jets of metrics at $p$ has dimension given by:
\[ \dim J_{k}(g):=\sum_{i=0}^k \frac{n(n+1)}{2} \dim (S^i T_p) =
\sum_{i=0}^{k} \frac{n(n+1)}{2}\binom{ n + i - 1}{i}. \]
Thus:
\[
\dim J_{k}(g) - \dim J_{k+2}(x,\phi) =
(n+1)(a_{k,n} - b_{k,n} ) \]
where
\begin{eqnarray*}
a_{k,n} & := & \left( \frac{n}{2}-1 \right) \sum_{i=1}^k  \binom{n+1-i}{i}, \\
b_{k,n} & := & \binom{n+k}{k+1} + \binom{n+k+1}{k+2}.
\end{eqnarray*}
For fixed $n>2$, $a_{k,n}$ grows as order $k^n$ whereas $b_{k,n}$ grows as order $k^{n-1}$. So for sufficiently large $k$, $\dim J_{k}(g) > \dim J_{k+2}(x,\phi)$.

It follows that if $n>2$, for sufficiently large $k$, the generic $k$-jet of a metric tensor does not admit any compatible dually flat structure no matter how one extends this $k$-jet to a smooth metric.
\end{proof}

This counting argument can be summarized by saying that the number of metrics depends upon $\frac{1}{2}n(n+1)$ functions of $n$ variables, whereas the data for a Hessian structure depends upon only $n+1$ functions of $n$ variables \cite{bryant}. Our dimension counting merely makes this argument precise.

\section{A curvature obstruction in dimensions $\geq 4$}
\label{section:curvature}

Our aim in this section is to find more concrete obstructions to the existence of Hessian metrics. The key results are the following  \cite{shima}:

\begin{proposition}
Let $(M,g)$ be a Riemannian manifold.
Let $\D$ denote the Levi--Civita connection and let $\ol{\D} = \D + A$ be a $g$-dually flat connection. Then
\begin{enumerate}[(i)]
\item
The tensor $A_{ijk}$ lies in $S^3 T^*$. We shall call it the {\em $S^3$-tensor} of $\ol{\D}$.
\item
The $S^3$-tensor determines the Riemann curvature tensor as follows:
\begin{equation}
\label{eqn:curvatureFromS3}
R\indices{_{ijkl}} = -g^{ab} A\indices{_{ika}}A\indices{_{jlb}}
+ g^{ab} A\indices{_{ila}}A\indices{_{jkb}}.
\end{equation}
\end{enumerate}
\end{proposition}
\begin{proof}
$A \in T^* \ot T^* \ot T$. The condition that $\ol{\D}$ is torsion free is equivalent to requiring that $A \in S^2 T^* \ot T$. Using the metric to identify $T$ and $T^*$, the condition that $\ol{\D}$ is dually torsion free can be written as $A \in S^3 T^*$. 

Expanding the formula $\ol{R}_{XY}Z = \ol{\D}_X \ol{\D}_Y Z - \ol{\D}_Y \ol{\D}_X - \ol{\D}_{[X,Y]}Z$ in terms of $\D$ and $A$, one obtains the following curvature identity:
\begin{equation}
\label{eqn:coordinateFreeCurvature}
\ol{R}_{XY} Z = R_{XY} Z + 2( \D_{[X}A)_{Y]}Z + 2 A_{[X}A_{Y]}Z. \end{equation}
Here $\ol{R}=0$ is the curvature of $\ol{\D}$ and the square brackets denote anti-symmetrization. Since $\ol{\D}$ is dually flat $\ol{R}=0$.

Continuing to use the metric to identify $T$ and $T^*$, the symmetries of the curvature tensor tell us that $R \in \Lambda^2 T \ot \Lambda^2 T$. On the other hand, $(\D_{[\cdot}A)_{\cdot]} \in \Lambda^2 T \ot S^2 T$. Thus if one projects equation
\eqref{eqn:coordinateFreeCurvature} onto $\Lambda^2 T \ot \Lambda^2 T$ one obtains the curvature identity \eqref{eqn:curvatureFromS3}.
\end{proof}
We define a quadratic equivariant map $\rho$ from $S^3 T^* \longrightarrow \Lambda^2 T^* \ot \Lambda^2 T^*$ by:
\[ \rho( A_{ijk} ) = -g^{ab} A\indices{_{ika}}A\indices{_{jlb}}
+ g^{ab} A\indices{_{ila}}A\indices{_{jkb}} \]
\begin{corollary}
In dimensions $>4$ the condition that $R$ lies in the image of $\rho$ gives a non-trivial necessary condition for a metric $g$ to be a Hessian metric.
\end{corollary}
\begin{proof}
$\dim S^3 T = \binom{ n+2 }{n-1} = \frac{1}{6}n(1+n)(2+n)$. The dimension of the space of algebraic curvature tensors, ${\cal R}$, is $\dim {\cal R}=\frac{1}{12}n^2(n^2-1)$. So ${\dim \cal R} - \dim S^3 T = \frac{1}{12}n(n-4)(1+n)^2$. This is strictly positive if $n>4$.
\end{proof}

Rather more surprisingly, the condition that $R$ lies in the image of $\rho$ gives a non-trivial condition in dimension $4$. In dimension $4$, $\dim S^3 T = \dim {\cal R} = 20$ and yet we will see that the dimension of the image of $\rho$ is only $18$. This can be tested heuristically by computer experiment: pick a ``random'' tensor $A \in S^3 T^*$ and compute the rank of the derivative $\rho_*$ at $A$. The result is almost certain to be $18$, giving strong evidence for the claim.

To prove the claim rigorously we wish to write the curvature conditions in dimension $4$ in a more explicit form.

One possible approach to finding explicit curvature conditions is to use Gr\"obner bases. Let $\rho^\C: S^3 T^* \ot \C \longrightarrow \Lambda^2 \ot \Lambda^2 \ot \C$ be the complexification of the map $\rho$. One can think of the image of $\rho^\C$ as being a complex algebraic variety parameterized by $\rho^\C$. What we seek is a set of algebraic equations on the curvature that define this variety. In theory, therefore, one can then use the well known ``implicitization'' algorithm described in \cite{cox2007ideals} to find implicit equations for the image of $\rho^\C$. These implicit equations would be precisely the curvature conditions we seek. While this approach proves the existence of the desired curvature conditions, it does not perform sufficiently well in practice. For example we have attempted to implement this brute force strategy using Singular (together with some more or less obvious modifications to improve the efficiency) without success. 

An alternative strategy is to use the $\SO(n)$ equivariance of the problem to enumerate all the possible equivariant conditions on a given power of the curvature tensor and search, by brute force, amongst these conditions. Let us describe this approach in more detail.

Let ${\cal R}$ denote the space of algebraic curvature tensors. The $p$-th power of the curvature must lie in $S^p{\cal R}$. So an equivariant $p$-th order condition on the curvature must be given by an equivariant linear map $\phi: S^p({\cal R}) \longrightarrow V$ for some representation $V$ of $\SO(n)$. We search for these conditions systematically. For each $p$, we decompose $S^p(\cal R)$ into irreducible components under $\SO(n)$. For each irreducible representation $V$, let $m$ be the multiplicity with which it occurs in the decomposition into irreducibles. We can correspondingly define $m$ independent maps $\phi_1, \phi_2, \ldots \phi_m$ from $S^p(\cal R)$ to $V$. Pick a finite sequence of random tensors $\{t_i\} \in S^3 T^*$ so that the set of vectors $\{ \phi_1( \rho(t_i) ) \op \phi_2( \rho( t_i) ) \op \ldots \op \phi_m( \rho(t_i)) \}$ has the maximum possible rank: more precisely, keep adding random vectors to the sequence $\{t_i\}$ until the rank stops increasing. Write $r_i=\rho(t_i) \in {\cal R}$. Add algebraic curvature tensors $r_j$ to the finite sequence $\{ r_i \}$ so that the set of vectors $\{ \phi_1( r_j) \op \phi_2( r_j) \op \ldots \op \phi_m( r_j) \}$ has maximal rank.
It is now a simple matter of linear algebra to find the linear conditions satisfied by the terms $\{\phi_a( \rho(t_i) )\}$ that are not satisfied by the $\{ \phi_a( r_j ) \}$.

Because of the use of random generation of tensors, in theory, one cannot be certain that a curvature ``identity'' discovered in this way holds in general or that all curvature identities will be found by this approach. However, once one has identified a candidate curvature identity, it is easy to verify it by brute force symbolic algebra. It is also easy to tell one has identified all the curvature identities implied by the condition $R \in \rho^\C$ by dimension counting.

While, this brute force algorithm is not particularly pretty it is effective --- at least in dimension $4$. Since $\hbox{Spin}(4) = \SU(2) \times \SU(2)$ the representation theory of $\SO(4)$ is quite simple, so the above strategy is not hard to implement. We have not attempted to implement the strategy in higher dimensions.

Here is the result.

\begin{theorem}The space of possible curvature tensors for
a Hessian $4$-mani\-fold is $18$ dimensional. In particular the
curvature tensor must satisfy the identities:
\begin{equation}
\label{eqn:pontryagin1}
\alpha( R\indices{_{ija}^b}R\indices{_{klb}^a} ) = 0
\end{equation}
\begin{equation}
\label{eqn:cubic}
\alpha( R\indices{_{iajb}}R\indices{_{k}^{b}_{cd}}R\indices{_{l}^{dac}} - 2R\indices{_{iajb}}R\indices{_{kc}^{a}_{d}}R\indices{_{l}^{dbc}} ) = 0 
\end{equation}
where $\alpha$ denotes antisymmetrization of the $i$, $j$, $k$ and $l$ indices.
\end{theorem}
\begin{proof}
Using a symbolic algebra package, write the general tensor in $S^3 T^*$ with respect to an orthonormal basis in terms of its $20$ components. Compute the curvature tensor using equation \eqref{eqn:curvatureFromS3}. One can then directly check the above identities.
\end{proof}
In fact, equation \eqref{eqn:pontryagin1} is simple to prove by hand. Moreover, it generalizes to higher dimensions. To make the proof as vivid as possible, we introduce a graphical notation that simplifies manipulating symmetric powers of the $S^3$-tensor $A$ (this is based on the notation given in the appendix of \cite{penroseAndRindler}). When using this notation we will always assume that our coordinates are orthonormal at the point where we perform the calculations so we can ignore the difference between upper and lower indices of ordinary tensor notation.

Given a tensor defined by taking the $n$-th tensor power of the $S^3$-tensor tensor $A$ followed by a number of contractions we can define an associated graph by:
\begin{itemize}
\item Adding one vertex to the graph for each occurrence of $A$;
\item Adding an edge connecting the vertices for each contraction between the vertices;
\item Adding a vertex for each tensor index that is not contracted and labelling it with the same symbol used for the index. Join this vertex to the vertex representing the associated occurrence of $A$.
\end{itemize}

When two tensors written in the Einstein summation convention are juxtaposed in a formula, we will refer to this as ``multiplying'' the tensors. This multiplication corresponds graphically to connecting labelled vertices of the graphs according to the contractions that need to be performed when the tensors are juxtaposed. Since this multiplication is commutative, and since the $S^3$-tensor is symmetric, one sees that there is a one to one correspondence between isomorphism classes of such graphs and equivalently defined tensors.

We can use these graphs in formulae as an alternative notation for the tensor represented by the graph. For example, we can write the curvature identity \eqref{eqn:curvatureFromS3} graphically as
\begin{equation}
\label{graphicalCurvatureIdentity}
R_{ijkl} =
-
\scriptsize{
\raisebox{-0.5\height}{
\begin{tikzpicture}[scale=0.5]
\node (i) at (0,2) {$i$} ;
\node (j) at (2,2) {$j$};
\node (k) at (0,0) {$k$};
\node (l) at (2,0) {$l$};
\coordinate (p1) at (0,1) {};
\coordinate (p2) at (2,1) {};
\draw [-] (i) -- (p1);
\draw [-] (k) -- (p1);
\draw [-] (j) -- (p2);
\draw [-] (l) -- (p2);
\draw [-] (p1) -- (p2);
\end{tikzpicture}
}
+
\raisebox{-0.5\height}{
\begin{tikzpicture}[scale=0.5]
\node (i) at (0,2) {$i$} ;
\node (j) at (2,2) {$j$};
\node (k) at (0,0) {$k$};
\node (l) at (2,0) {$l$};
\coordinate (p1) at (0.7,0.7) {};
\coordinate (p2) at (1.3,0.7) {};
\draw [-] (i) -- (p2);
\draw [-] (l) -- (p2);
\draw [-][draw=white,double=black,very thick] (j) -- (p1);
\draw [-] (k) -- (p1);
\draw [-] (p1) -- (p2) ;
\end{tikzpicture}
}
}
.
\end{equation}
As an extended example, here is a graphical proof that the first Bianchi identity follows from equation \eqref{eqn:curvatureFromS3}:
\begin{multline*}
R_{ijkl} + R_{jkil} + R_{kijl} \\
\begin{aligned}
%
%
& =
-
\scriptsize{
\raisebox{-0.5\height}{  
\begin{tikzpicture}[scale=0.5]
\node (i) at (0,2) {$i$} ;
\node (j) at (2,2) {$j$};
\node (k) at (0,0) {$k$};
\node (l) at (2,0) {$l$};
\coordinate (a) at (1,1) {};
\coordinate (p1) at (0,1) {};
\coordinate (p2) at (2,1) {};
\draw [-] (i) -- (p1);
\draw [-] (k) -- (p1);
\draw [-] (j) -- (p2);
\draw [-] (l) -- (p2);
\draw [-] (p1) -- (a);
\draw [-] (p2) -- (a);
\end{tikzpicture}
}
+
\raisebox{-0.5\height}{
\begin{tikzpicture}[scale=0.5]
\node (i) at (0,2) {$i$} ;
\node (j) at (2,2) {$j$};
\node (k) at (0,0) {$k$};
\node (l) at (2,0) {$l$};
\coordinate (a) at (1,0.7) {};
\coordinate (p1) at (0.7,0.7) {};
\coordinate (p2) at (1.3,0.7) {};
\draw [-] (i) -- (p2);
\draw [-] (l) -- (p2);
\draw [-][draw=white,double=black,very thick] (j) -- (p1);
\draw [-] (k) -- (p1);
\draw [-] (p1) -- (a);
\draw [-] (p2) -- (a);
\end{tikzpicture}
}
-
\raisebox{-0.5\height}{ 
\begin{tikzpicture}[scale=0.5]
\node (i) at (0,2) {$j$} ;
\node (j) at (2,2) {$k$};
\node (k) at (0,0) {$i$};
\node (l) at (2,0) {$l$};
\coordinate (a) at (1,1) {};
\coordinate (p1) at (0,1) {};
\coordinate (p2) at (2,1) {};
\draw [-] (i) -- (p1);
\draw [-] (k) -- (p1);
\draw [-] (j) -- (p2);
\draw [-] (l) -- (p2);
\draw [-] (p1) -- (a);
\draw [-] (p2) -- (a);
\end{tikzpicture}
}
+
\raisebox{-0.5\height}{
\begin{tikzpicture}[scale=0.5]
\node (i) at (0,2) {$j$} ;
\node (j) at (2,2) {$k$};
\node (k) at (0,0) {$i$};
\node (l) at (2,0) {$l$};
\coordinate (a) at (1,0.7) {};
\coordinate (p1) at (0.7,0.7) {};
\coordinate (p2) at (1.3,0.7) {};
\draw [-] (i) -- (p2);
\draw [-] (l) -- (p2);
\draw [-][draw=white,double=black,very thick] (j) -- (p1);
\draw [-] (k) -- (p1);
\draw [-] (p1) -- (a);
\draw [-] (p2) -- (a);
\end{tikzpicture}
}
-
\raisebox{-0.5\height}{ 
\begin{tikzpicture}[scale=0.5]
\node (i) at (0,2) {$k$} ;
\node (j) at (2,2) {$i$};
\node (k) at (0,0) {$j$};
\node (l) at (2,0) {$l$};
\coordinate (a) at (1,1) {};
\coordinate (p1) at (0,1) {};
\coordinate (p2) at (2,1) {};
\draw [-] (i) -- (p1);
\draw [-] (k) -- (p1);
\draw [-] (j) -- (p2);
\draw [-] (l) -- (p2);
\draw [-] (p1) -- (a);
\draw [-] (p2) -- (a);
\end{tikzpicture}
}
+
\raisebox{-0.5\height}{
\begin{tikzpicture}[scale=0.5]
\node (i) at (0,2) {$k$} ;
\node (j) at (2,2) {$i$};
\node (k) at (0,0) {$j$};
\node (l) at (2,0) {$l$};
\coordinate (a) at (1,0.7) {};
\coordinate (p1) at (0.7,0.7) {};
\coordinate (p2) at (1.3,0.7) {};
\draw [-] (i) -- (p2);
\draw [-] (l) -- (p2);
\draw [-][draw=white,double=black,very thick] (j) -- (p1);
\draw [-] (k) -- (p1);
\draw [-] (p1) -- (a);
\draw [-] (p2) -- (a);
\end{tikzpicture}
} 
} \\
%
%
& =
-
\scriptsize{
\raisebox{-0.5\height}{  
\begin{tikzpicture}[scale=0.5]
\node (i) at (0,2) {$i$} ;
\node (j) at (2,2) {$j$};
\node (k) at (0,0) {$k$};
\node (l) at (2,0) {$l$};
\coordinate (a) at (1,1) {};
\coordinate (p1) at (0,1) {};
\coordinate (p2) at (2,1) {};
\draw [-] (i) -- (p1);
\draw [-] (k) -- (p1);
\draw [-] (j) -- (p2);
\draw [-] (l) -- (p2);
\draw [-] (p1) -- (a);
\draw [-] (p2) -- (a);
\end{tikzpicture}
}
+
\raisebox{-0.5\height}{
\begin{tikzpicture}[scale=0.5]
\node (i) at (0,2) {$i$} ;
\node (j) at (2,2) {$j$};
\node (k) at (0,0) {$k$};
\node (l) at (2,0) {$l$};
\coordinate (a) at (1,0.7) {};
\coordinate (p1) at (0.7,0.7) {};
\coordinate (p2) at (1.3,0.7) {};
\draw [-] (i) -- (p2);
\draw [-] (l) -- (p2);
\draw [-][draw=white,double=black,very thick] (j) -- (p1);
\draw [-] (k) -- (p1);
\draw [-] (p1) -- (a);
\draw [-] (p2) -- (a);
\end{tikzpicture}
}
-
\raisebox{-0.5\height}{  
\begin{tikzpicture}[scale=0.5]
\node (i) at (0,2) {$i$} ;
\node (j) at (2,2) {$j$};
\node (k) at (0,0) {$k$};
\node (l) at (2,0) {$l$};
\coordinate (a) at (1,1) {};
\coordinate (p1) at (1,2.0) {};
\coordinate (p2) at (1,0.0) {};
\draw [-] (i) -- (p1);
\draw [-] (k) -- (p2);
\draw [-] (j) -- (p1);
\draw [-] (l) -- (p2);
\draw [-] (p1) -- (a);
\draw [-] (p2) -- (a);
\end{tikzpicture}
}
+
\raisebox{-0.5\height}{
\begin{tikzpicture}[scale=0.5]
\node (i) at (0,2) {$i$} ;
\node (j) at (2,2) {$j$};
\node (k) at (0,0) {$k$};
\node (l) at (2,0) {$l$};
\coordinate (a) at (1,1) {};
\coordinate (p1) at (0,1) {};
\coordinate (p2) at (2,1) {};
\draw [-] (i) -- (p1);
\draw [-] (k) -- (p1);
\draw [-] (j) -- (p2);
\draw [-] (l) -- (p2);
\draw [-] (p1) -- (a);
\draw [-] (p2) -- (a);
\end{tikzpicture}
}
-
\raisebox{-0.5\height}{   
\begin{tikzpicture}[scale=0.5]
\node (i) at (0,2) {$i$} ;
\node (j) at (2,2) {$j$};
\node (k) at (0,0) {$k$};
\node (l) at (2,0) {$l$};
\coordinate (a) at (1,0.7) {};
\coordinate (p1) at (0.7,0.7) {};
\coordinate (p2) at (1.3,0.7) {};
\draw [-] (i) -- (p2);
\draw [-] (l) -- (p2);
\draw [-][draw=white,double=black,very thick] (j) -- (p1);
\draw [-] (k) -- (p1);
\draw [-] (p1) -- (a);
\draw [-] (p2) -- (a);
\end{tikzpicture}
}
+
\raisebox{-0.5\height}{
\begin{tikzpicture}[scale=0.5]
\node (i) at (0,2) {$i$} ;
\node (j) at (2,2) {$j$};
\node (k) at (0,0) {$k$};
\node (l) at (2,0) {$l$};
\coordinate (a) at (1,1) {};
\coordinate (p1) at (1,2.0) {};
\coordinate (p2) at (1,0.0) {};
\draw [-] (i) -- (p1);
\draw [-] (k) -- (p2);
\draw [-] (j) -- (p1);
\draw [-] (l) -- (p2);
\draw [-] (p1) -- (a);
\draw [-] (p2) -- (a);
\end{tikzpicture}
}
} \\
& = 0.
\end{aligned}
\end{multline*}
The last line follows because the H's, I's and X's cancel. 

With this notation understood, let us now prove a generalization of equation \eqref{eqn:pontryagin1}.

\begin{theorem}
The Pontryagin forms (i.e.\ the closed forms given as polynomials in the curvature tensor that represent the Pontryagin classes) vanish on a Hessian manifold.
\end{theorem}
\begin{proof}
We need to briefly recall the theory of characteristic classes in order give a more formal definition of the Pontryagin forms. See \cite{milnorandstasheff} for a less brisk account.

An invariant polynomial $P$ defined on $n \times n$ complex matrices is a polynomial in the coefficients of the matrix that satisfies $P(TXT^{-1})=P(X)$ for every non singular matrix $T$. Given a complex $k$ dimensional vector bundle $V$ with a local trivialization $s^1, s^2, \ldots s^k$ and a connection $\D$ we can write the curvature of $\D$ as $F s^i = \Omega_{ij} s^j$ where each $\Omega_{ij}$ is a complex valued two form. Since the algebra of even degree forms is commutative, it makes sense to evaluate the polynomial $P$ on $\Omega$. The result is a form of degree $2k$ which we will denote as $P(F)$. The requirement that $P$ is an invariant polynomial ensures that this definition is independent of the choice of trivialization. The key results in the theory of characteristic classes are that P(F) is closed and that the de Rham cohomology class of P(F) is independent of the choice of connection. We will refer to this cohomology class as the characteristic class associated with P.

Let $P_k$ denote the invariant polynomial given by the
coefficient of $t^k$ in the expansion of $\det(I + tA)$.
If we understand the eigenvalues of a matrix to mean the diagonal entries of a conjugate matrix in Jordan normal form, then the polynomial $P_k$ is associated with the $k$-th elementary symmetric function in the eigenvalues. 

The associated characteristic class is related to the $k$-th Chern class $c_k$ of the bundle $V$ by:
\[ [P_k(F)] = (2 \pi i)^k c_k(V). \]
For our purposes we could take this as the definition of the Chern class. 

The Pontryagin classes, $p_k$, of a real vector bundle $E$ are defined in terms of the Chern classes of the complexificiation:
\[ p_k(E) = (-1)^k c_{2k}( E \ot \C). \]
The Pontryagin classes of a manifold are the Pontryagin classes of
the complexified tangent space.

This completes our review of the theory of characteristic classes.
We can now define the {\em Pontryagin forms} of a Riemannian manifold $(M,g)$ to be the forms $P_{\sigma_{2k}}(R)$ where $R$ is the Levi--Civita connection.

Let $X\indices{_a^b}$ be any matrix. We define an invariant polynomial $Q_k$ by $k$ contractions:
\begin{eqnarray*}
Q_k(X) & = & X\indices{_{a_1}^{a_2}} X\indices{_{a_2}^{a_3}} X\indices{_{a_3}^{a_4}} \ldots X\indices{_{a_{k-1}}^{a_k}} X\indices{_{a_{k}}^{a_1}} \\ 
& = & \lambda_1^k + \lambda_2^k + \ldots + \lambda_n^k
\end{eqnarray*}
where the $\lambda_i$ are the eigenvalues of $X$.

The general theory of symmetric polynomials shows that any symmetric polynomial can be written as as sum and product of the $Q_k$. Thus all the Chern forms, $P_k(F)$, are generated by sums and products of the forms $Q_k(F)$.

Thus to prove our result we need to show that the tensor
\begin{multline*}
Q^p_{i_1 i_2 \ldots i_{2p}} = \\
\sum_{\sigma \in S_{2p}} \sgn( \sigma )
R\indices{_{i_{\sigma(1)}{i_{\sigma(2)}}{a_1}}^{a_2}}
R\indices{_{i_{\sigma(3)}{i_{\sigma(4)}}{a_2}}^{a_3}}
R\indices{_{i_{\sigma(5)}{i_{\sigma(6)}}{a_3}}^{a_4}}
\ldots
R\indices{_{i_{\sigma(2p-1)}{i_{\sigma(2p)}}{a_p}}^{a_1}}
\end{multline*}
vanishes on a Hessian manifold.

We can rewrite the curvature identity \eqref{eqn:curvatureFromS3} as:
\[ R_{i_1i_2ab} = \sum_{\sigma \in S_{2}} -\sgn( \sigma )
\scriptsize{
\raisebox{-0.5\height}{
\begin{tikzpicture}[scale=0.5]
\node (i) at (0,2) {$i_{\sigma(1)}$} ;
\node (j) at (2,2) {$i_{\sigma(2)}$};
\node (k) at (0,0) {$a$};
\node (l) at (2,0) {$b$};
\coordinate (p1) at (0,1) {};
\coordinate (p2) at (2,1) {};
\draw [-] (i) -- (p1);
\draw [-] (k) -- (p1);
\draw [-] (j) -- (p2);
\draw [-] (l) -- (p2);
\draw [-] (p1) -- (p2);
\end{tikzpicture}
}
}.
\]
Thus we can replace each $R$ in the formula for $Q^p$ with an `H'. The legs of adjacent H's are then connected. The result is:
\begin{multline*}
Q^p_{i_1 i_2 \ldots i_{2p}} = \\
(-1)^p \sum_{\sigma \in S_{2p}}
\sgn(\sigma)
\scriptsize{
\raisebox{-0.5\height}{
\begin{tikzpicture}[scale=0.5]
\node (i1) at (0,2) {$i_{\sigma(1)}$};
\coordinate (p1) at (0,1) {};
\node (i2) at (2,2) {$i_{\sigma(2)}$};
\coordinate (p2) at (2,1) {};
\node (i3) at (4,2) {$i_{\sigma(3)}$};
\coordinate (p3) at (4,1) {};
\node (i4) at (6,2) {$i_{\sigma(4)}$};
\coordinate (p4) at (6,1) {};
\node (i5) at (8,2) {$i_{\sigma(5)}$};
\coordinate (p5) at (8,1) {};
\node (i6) at (10,2) {$i_{\sigma(6)}$};
\coordinate (p6) at (10,1) {};
\node (p7) at (12,1) {\ldots};
\node (i2nm1) at (14,2) {$i_{\sigma(2p-1)}$};
\coordinate (p2nm1) at (14,1) {};
\node (i2n) at (16,2) {$i_{\sigma(2p)}$};
\coordinate (p2n) at (16,1) {};
\draw [-] (i1) -- (p1);
\draw [-] (i2) -- (p2);
\draw [-] (i3) -- (p3);
\draw [-] (i4) -- (p4);
\draw [-] (i5) -- (p5);
\draw [-] (i6) -- (p6);
\draw [-] (i2nm1) -- (p2nm1);
\draw [-] (i2n) -- (p2n);
\draw [-] (p1) -- (p2);
\draw [-] (p2) -- (p3);
\draw [-] (p3) -- (p4);
\draw [-] (p4) -- (p5);
\draw [-] (p5) -- (p6);
\draw [-] (p6) -- (p7);
\draw [-] (p7) -- (p2nm1);
\draw [-] (p2nm1) -- (p2n);
\draw [-] (0,1) arc (-180:-90:1) -- (15,0) arc (-90:0:1);
\end{tikzpicture}
}
}.
\end{multline*}
Since the cycle $1 \rightarrow 2 \rightarrow 3 \ldots \rightarrow 2p \rightarrow 1$ is an odd permutation, one sees that $Q^p=0$.
\end{proof}

The Pontryagin classes, therefore, give a topological obstruction to the existence of a Hessian metric in dimensions $\geq 4$. Note that it is possible to manufacture topologically interesting manifolds that admit Hessian metrics. This is true because: all analytic $2$-manifolds are Hessian (as we will prove later); all products of Hessian manifolds are Hessian manifolds; all hyperbolic manifolds are Hessian.

The graphical notation is not essential to proving the above result. Nevertheless we find it illuminates the proof. We note that similar graphs have been used to good effect in papers such as \cite{hitchin2001curvature} and \cite{garoufalidis1998some}.

While equation \eqref{eqn:pontryagin1} generalizes easily to higher dimensions, equation \eqref{eqn:cubic} does not hold in dimensions $\geq 5$. To see this, simply pick a random tensor in $S^3 T^*$ for $n=5$ and numerically check whether \eqref{eqn:cubic} holds: in all probability it will not. Since the equation does not hold in higher dimensions, the graphical proof technique cannot work in this case.

Note that we do not claim that equations \eqref{eqn:pontryagin1} and \eqref{eqn:cubic} provide a sufficient condition for a $4$-d curvature tensor to be the curvature tensor of some Hessian manifold. We have only considered the complex algebraic map $\rho^\C$. In real algebraic geometry, one cannot expect to translate a parametric representation of a variety into a set of implicit equations, one can only expect to find a set of implicit inequalities. Algorithms do exist \cite{basu} to solve such problems but they are not particularly efficient.

Although we cannot prove our conditions are sufficient in dimension $4$, we can at least prove the corresponding result in dimension $3$.

\begin{theorem}In $3$ dimensions, all possible Riemann curvature
tensors occur as the curvature tensor of a Hessian metric.
\end{theorem}
\begin{proof}

First note that given a tensor $A \in S^3 T^*_p$ we can find a Hessian metric whose $S^3$-tensor at $p$ is given by $A$. This true because the $S^3$-tensor is determined (up to raising and lowering of indices and a constant) by $\partial_i \partial_j \partial_k \phi$ where $\phi$ is the potential defining the Hessian metric. So to prove the theorem we simply need to show that $\rho$ is surjective in dimension $3$.
Dimension counting suggests that this is likely to be true but does not give a proof.

In $3$ dimensions, the curvature tensor is determined entirely
by the Ricci tensor. Since any symmetric matrix can be diagonalized using orthogonal transformations, we can choose an orthonormal basis $\{ e_1, e_2, e_3 \}$ for the cotangent space such that the Ricci tensor is given by $r = \lambda_1 e_1^2 + \lambda_2 e_2^2 + \lambda_3 e_3^2$.

We define $\rho_2$ to be the composition of $\rho$ with the contraction $R\indices{_{ijk}^l} \rightarrow R\indices{_{iak}^a}$ used to define the Ricci tensor.

Now consider symmetric tensors $A$ of the following form:
\[ A = a_1 e_1^3 + a_2 e_2^3 + a_3 e_3^3 + b_{13} e_1^2 e_3 + e_2^2 e_1 + e_2^2 e_3 + b_{31} e_3^2 e_1. \]
A straightforward computation (which it is easiest to get a computer to perform) shows that for such an $A$ we have:
\[ \rho_2( A ) = \alpha e_1^2 + \beta e_2^2 + \gamma e_3^2 + \delta e_1 e_3 \]
where
\begin{eqnarray*}
\alpha & = & 8 (1 - (1 + 3 a_3) b_{13} + b_{13}^2 + b_{31}^2 - 3 a_1 (1 + b_{31})) \, , \\
\beta & = & -8 (-2 + 3 a_1 + 3 a_3 + b_{13} + b_{31}) \, , \\
\gamma & = & 8 (1 + b_{13}^2 - 3 a_3 (1 + b_{13}) - b_{31} - 3 a_1 b_{31} + b_{31}^2) \, , \\
\delta & = & -8 (-1 + b_{13} + b_{31}) \, .
\end{eqnarray*}
Notice that $\alpha - \gamma$ does not contain any quadratic terms.
So to find $A$ of this form with $\rho_2(A)$ equal to $r$ we must solve the three linear equations $\alpha - \gamma = \lambda_1 - \lambda_3$, $\beta=\lambda_2$ and $\delta=0$ and the quadratic equation $\alpha = \lambda_1$. If one first solves the linear equations to find expressions for $a_1$, $b_{13}$ and $b_{31}$ in terms of the $\lambda_i$ and $a_3$ one can then use the quadratic equation to find an expression for $a_3$ in terms of the $\lambda_i$. In fact a cancellation occurs in the quadratic terms and one obtains a linear equation for $a_3$.

The end result is that, so long as $\lambda_1 \neq \lambda_3$ we always have a solution of the form:
\begin{eqnarray*}
a_3 & = & \frac{\lambda_1^2 + 16 \lambda_2 - \lambda_1 \lambda_2 - 16 \lambda_3 - 
    2 \lambda_1 \lambda_3 + \lambda_2 \lambda_3 + 
    \lambda_3^2}{48 (\lambda_1 - \lambda_3)} \\
a_1 &=& \frac{1}{24} (8 - 24 a_3 - \lambda_2) \\
b_{13} &=& 3 a_3 +\frac{1}{16}( - \lambda_1 + \lambda_2 + \lambda_3) \\
b_{31} &=& 1 - 3 a_3 + \frac{1}{16}(\lambda_1 - \lambda_2 - \lambda_3)
\end{eqnarray*}

The only case that is not covered by this result is when the Ricci tensor is diagonal. In this case it is easy to show that one can take $A$ defined by:

\[ A  = \frac{20 - \lambda}{48} e_1^3 + e_2^2 e_1 + \frac{4-\lambda}{16} e_3^2 e_1 + e_1 e_2 e_3 \]

\end{proof}

We end this section on curvature identities by raising some questions.
\begin{enumerate}[(i)]
\item Can one find a short proof of equation \eqref{eqn:cubic} that does not require any use of a computer? 
\item Can one efficiently find all the explicit curvature conditions that must be satisfied by a Hessian metric in a fixed dimension $n \geq 5$? Can one find all the curvature conditions that hold for all $n$?
\item For large enough $n$, is the condition that the curvature lies in the image of $\rho$ a sufficient condition for a metric to be Hessian?
\end{enumerate}

\section{All analytic 2 metrics are Hessian}
\label{section:cartanKahler}

In this section we will prove that all analytic 2 manifolds $(M,g)$
are Hessian. This result has been obtained independently by Bryant \cite{bryant}. The proof is an application of Cartan--K\"ahler theory. See \cite{bryant1991exterior} or \cite{seiler} for an overview of Cartan--K\"ahler theory. Our presentation is closer to that of Seiler.

Let $V$ and $W$ be vector bundles over an n-dimensional manifold $M^n$ and
let $D:\Gamma(V) \longrightarrow \Gamma(W)$ be an order $k$ differential operator mapping sections of $V$ to sections of $W$. Equivalently, $D$ is a mapping from $k$-jets of $V$ at $p$ to elements of $W_p$.

Recall that one has the exact sequence $0 \longrightarrow S^k T^* \ot V \longrightarrow J_{k}(V) \longrightarrow J_{k-1}(V) \longrightarrow 0$. This exact sequence is a consequence of the fact that derivatives commute. It tells us that a $k$-jet is determined by a $(k-1)$-jet together with an element of the $k$-th symmetric power.

As a result of this exact sequence, the highest order terms of the differential operator $D$ defines a map
$\sigma: S^k T^*_p \ot V_p \longrightarrow W_p$ called the {\em symbol} of the differential operator. We will assume from now on that the differential operator is quasilinear so the symbol is a linear map.

Note that if the symbol is onto then any $(k-1)$-jet can be extended to a $k$-jet solution of the equation $D(v)=w$. We will generalize this observation.

By differentiating the equation $D(v)=w$ one can obtain a $(k+1)$-th order differential equation. The top order term of this equation defines the first
prolongation of the symbol $\sigma_1$. This is a map $\sigma_1: S^{k+1} T^* \ot V \longrightarrow T^* \ot W$. More generally, one can differentiate the equation $i$ times to get the $i$-th prolongation of the symbol $\sigma_i:S^{k+i} T^* \ot V \longrightarrow S^i T^* \ot W$. If all of the first $i$ prolongations of the symbol are onto then any $k$-th jet solution of the differential equations can be extended to a $(k+i)$-th jet solution.

On the  other hand, if a particular prolongation of the symbol is not onto then this indicates that one may have found an obstruction to the local existence of
solutions to the differential equation. The calculation in the previous section fits this pattern: the curvature identities we used are all consequence of the fact that derivatives commute, equivalently they are algebraic consequences of the fact that the symbol acts on the symmetric power of $T^*$.

When one wishes to prove that $\sigma_i$ is onto
for all $i$ one can use {\em Cartan's test} which we will now describe. Given a differential
equation as above and a basis $\{v_1, v_2, \ldots v_n\}$ for $T^*M$, define the map:
\[\sigma_{i,m}: S^{k+i} \langle v_1, v_2, \ldots v_m \rangle \ot V_p \longrightarrow S^i T^*_p \ot W_p\]
to be the restriction of $\sigma_i$.  Define $g_{i,m}:= \dim \ker
\sigma_{i,m}$. If one can find a basis $\{v_1, v_2, \ldots v_n\}$
and a number $\alpha$ such that $\sigma_i$ is onto for all $i \leq\alpha$ and such that $g_{\alpha,n} = \sum_{\beta=0}^k g_{\alpha-1
\beta}$ then the differential equation is said to be involutive.
It turns out that this implies that $\sigma_{\alpha+i}$ is onto for all $i$. If one
is working in the analytic category, one can then prove that
solutions to the differential equation exist
\cite{goldschmidt1967b, guilleminSternberg,cartan}.

We illustrate these ideas by proving that any analytic Riemannian metric on a $2$-manifold locally admits a $g$-dually torsion free flat connection --- in other words it is Hessian.

\begin{theorem}
Any analytic Riemannian metric on a $2$-manifold locally admits a $g$-dually torsion free flat connection.
\end{theorem}
\begin{proof}

The existence of such a connection is equivalent to finding a tensor $A \in T^* \ot T^* \ot T$ with:
\begin{enumerate}[(i)]
\item $R_{X,Y}Z + 2 \ol{\D}_{[X} A_{Y]} Z + 2 A_{[X}A_{Y]} Z = 0$ for all vectors $X$, $Y$ and $Z$.
\item $\iota(A) \in S^3 T^*$ where $\iota:T^* \ot T^* \ot T \longrightarrow T^* \ot T^* \ot T^*$ is the isomorphism determined by raising the final index using the metric.
\end{enumerate}
It is well known \cite{amari} that if these conditions hold then $\ol{\D}$ will also be torsion free and flat.

We choose an analytically varying orthonormal basis $\{e^1, e^2\}$ for $T^*$
with dual basis $\{e_1,e_2\}$. We write tensors with respect to this basis using indices and we will use the Einstein summation convention. So for example we have $R = R\indices{_{ijk}^l} \ot e^i \ot e^j \ot e^k \ot e_l$. In this index notation, we can write our curvature condition as follows:
\[
R\indices{_{ijk}^l} + 2 \D_{[i} A\indices{_{j]k}^l}  + A\indices{_{i \alpha}^l} A\indices{_{jk}^\alpha}  A\indices{_{j \alpha}^l}A\indices{_{ik}^\alpha} = 0.
\]
If we raise the $k$ index using the metric and then antisymmetrize over $k$ and $l$ we get
\begin{equation}
R\indices{_{ij}^{kl}} + A\indices{_{i \alpha}^l} A\indices{_{j}^{k \alpha}}
- A\indices{_{j \alpha}^{l}} A\indices{_{i}^{k \alpha}}
         - A\indices{_{i \alpha}^k} A\indices{_{j}^{l\alpha}}
         + A\indices{_{j \alpha}^k} A\indices{_{i}^{l\alpha}} = 0.
\label{eqn:curvatureCondition}         
\end{equation}
Symmetrizing over $k$ and $l$ on the other hand yields:
\begin{equation}
\D\indices{_{[i}} A\indices{_{j]}^{(kl)}} = 0
\label{eqn:derivativeCondition}         
\end{equation}
where the parentheses indicate symmetrization.

In two dimensions, the space $\Lambda^2 T \ot (\Lambda^2 T)^* $ is one dimensional, so the map sending a tensor with indices $B\indices{_{ij}^{kl}}$ to $B\indices{_{ab}^{ab}}$ is an isomorphism. Applying this isomorphism to
equation \eqref{eqn:curvatureCondition} yields the equivalent condition:
\begin{equation}
s - 2 A\indices{_{i \alpha}^j} A\indices{_{j}^{i \alpha}}
- 2 A\indices{_{j \alpha}^{j}} A\indices{_{i}^{i \alpha}} = 0.
\label{eqn:curvatureCondition2}         
\end{equation}
where $s=R\indices{_{ij}^{ij}}$ is the scalar curvature. The condition $\iota(A) \in S^3 T^*$ allows us to write:
\begin{eqnarray*}
 \iota(A) & = & a \, e^1 \odot e^1 \odot e^1 + b \, e^1 \odot e^1 \odot e^2 \\
 && {}+ c\, e^1 \odot e^2 \odot e^2 + d \, e^2 \odot e^2 \odot e^2
 \end{eqnarray*}
for some functions $a$, $b$, $c$ and $d$. We can then rewrite equation \eqref{eqn:curvatureCondition2} as:
\[  s - \frac{4}{9}( 3 a c + 3 b d - b^2 - c^2 ) = 0 \]
If we assume that $c \neq 0$ we can solve this last equation to compute $a$ in terms of $b$, $c$, $d$ and $s$. This allows us to write the equations for a $g$-dually torsion free flat connection with $c\neq 0$ as a differential equation in the three real functions $b$, $c$ and $d$. Equation \eqref{eqn:curvatureCondition} will be automatically satisfied. So we need only consider equation \eqref{eqn:derivativeCondition}. Because of its symmetries this has three independent components. If we write $x_i$ for the derivative of a function $x$ in the direction $e_i$ then the three components of \eqref{eqn:derivativeCondition}
can be written to highest order as:
\begin{eqnarray*}
-3 a_2 + b_1 = \alpha b_2 + \beta c_2 + \gamma d_2 + b_1 & = & \hbox{ terms without derivatives} \\
-b_2 + c_1 & = & \hbox{ terms without derivatives } \\
-c_2 + 3 d_1 & = & \hbox{ terms without derivatives } \\
\end{eqnarray*}
Here $\alpha$, $\beta$ and $\gamma$ are functions depending upon $s$, $b$, $c$, and $d$. The precise formulae are not important for our purposes.

If we think of $(b,c,d)$ as defining a section of the trivial $\R^3$
bundle, the condition that $(b,c,d)$ defines a $g$-dually flat
connection is a first order differential equation with symbol $\sigma:T+p^* \ot \R^3 \longrightarrow \R^3$ given by:
\[ \sigma = \left( \begin{array}{c c c c c c}
1 & 0 & 0 & \alpha & \beta & \gamma \\
0 & 1 & 0 & -1 & 0 & 0 \\
0 & 0 & 3 & 0 & -1 & 0
\end{array} \right) .
\]
Here we have used the standard basis $\{v_1, v_2, v_3\}$ for $\R^3$ and
the basis $\{e^1 \ot v^1, e^1 \ot v^2,  e^1 \ot v^3, 
e^2 \ot v^1, e^2 \ot v^2,  e^2 \ot v^3 \}$ for $T^* \ot \R^3$.

We conclude that $\sigma$ has rank 3. Since $\sigma$ is onto there is no obstruction to extending any $0$-jet $(b,c,d)$ with $c \neq 0$ to a $1$-jet solution of our differential equation.

With respect to the basis $\{ e^1, e^2 \}$ for $T^*$ the matrix
of $\sigma_{0,1}$ is simply:
\[ \sigma = \left( \begin{array}{c c c}
1 & 0 & 0  \\
0 & 1 & 0  \\
0 & 0 & 3 
\end{array} \right). \]
So $g_{0,1}=0$ and $g_{0,2}=3$.

The first prolongation of the symbol, $\sigma_1$ is similarly easy to calculate. 
$\sigma_1: S^2 T^* \longrightarrow T^* \ot \R^3$ is given by:
\[
\sigma_1 = \left(
\begin{array}{ccccccccc}
 1 & \alpha & 0 & 0 & \beta & 0 & 0 & \gamma & 0 \\
 0 & -1 & 0 & 1 & 0 & 0 & 0 & 0 & 0 \\
 0 & 0 & 0 & 0 & -1 & 0 & 3 & 0 & 0 \\
 0 & 1 & \alpha & 0 & 0 & \beta & 0 & 0 & \gamma \\
 0 & 0 & -1 & 0 & 1 & 0 & 0 & 0 & 0 \\
 0 & 0 & 0 & 0 & 0 & -1 & 0 & 3 & 0 \\
\end{array}
\right)
\]
This matrix is written with respect to the basis $\{ e^1 \odot e^1 \ot v^1, e^1 \odot e^1 \ot v^2,  e^1 \odot e^1 \ot v^3, 
e^1 \odot e^2 \ot v^1, e^1 \odot e^2 \ot v^2,  e^1 \odot e^2 \ot v^3,  e^2 \odot e^2 \ot v^1,   e^2 \odot e^2 \ot v^2,   e^2 \odot e^2 \ot v^3 \}$ for $S^2 T^* \ot \R^3$ and the same basis for $T^* \ot \R$ that we used earlier.

By permuting columns, one can transform this matrix for $\sigma_1$ into echelon form. Thus it has rank $6$ irrespective of $\alpha$, $\beta$ and $\gamma$. So $g_{0,1} + g_{0,2} = g_{1,2}$. The result now follows by Cartan's test.

\end{proof}

\bibliographystyle{alpha}
\bibliography{dualflat}

\end{document}